\documentclass[11pt]{amsart}

\usepackage[all]{xy}

\usepackage{amssymb,graphicx,xspace}

\newcommand*{\C}{\ensuremath{\mathbb C}\xspace}

\newcommand*{\M}{\ensuremath{\mathcal M}\xspace}

\newcommand*{\PP}{\ensuremath{\mathbb P}\xspace}

\newcommand*{\Tst}{\ensuremath{T_{\mathrm{st}}}\xspace}

\DeclareMathOperator{\card}{card}

\DeclareMathOperator{\Jac}{Jac}

\DeclareMathOperator{\ord}{ord}
\DeclareMathOperator{\trdeg}{tr.deg}

\let\eps\varepsilon
\let\ph\varphi

\newcommand*{\lst}[3][1]{\ensuremath{#2_{#1}, \ldots, #2_{#3}}\xspace}

\newtheorem{proposition}{Proposition}[section]
\newtheorem{theorem}[proposition]{Theorem}

\newtheorem{lemma}[proposition]{Lemma}

\theoremstyle{remark}
\newtheorem{note}[proposition]{Remark}

\theoremstyle{definition}
\newtheorem{definition}[proposition]{Definition}

\newtheorem{Not}[proposition]{Notation}

\author{Serge Lvovski}

\address{National Research University Higher School of Economics,
  Russian Federation 
  \hfil\break
Scientific Research Institute for System Analysis of the National
Research Centre ``Kurchatov Institute''}
\email{lvovski@gmail.com}

\title[Stabilization of fields of meromorphic functions]{Stabilization
  of fields of meromorphic functions on neighborhoods of a rational curve}



\thanks{This study was partially supported by the HSE
  University Basic Research Program and by the project FNEF-2024-0001
(Reg. no. 1023032100070-3-1.2.1)}

\begin{document}

\begin{abstract}
Suppose that $F$ is a smooth and connected complex surface (not
necessarily compact) containing a smooth rational curve~$C$ with positive
self-intersection. We prove that there exists a neighborhood $U\supset
C$ such that any meromorphic function defined on a connected
neighborhood of~$C$ in~$U$ can be extended to a meromorphic function
on the entire~$U$.  
\end{abstract}

\maketitle

\section{Introduction}

Suppose that $F$ is a smooth and connected complex surface containing
a curve $C\cong\PP^1$ with positive self-intersection ($\PP^1$ is the
Riemann sphere). It was shown in \cite{field} that the field of
meromorphic functions $\M(F)$ is isomorphic to either~\C, or the
field~$\C(X)$ of rational functions in one variable, or the
field~$\C(X,Y)$ of rational functions. If $U\supset V\supset C$ are
connected neighborhoods of~$C$, then $\M(U)$ is naturally embedded in
$\M(V)$.

In this paper we prove that the fields $\M(U)$ stabilize as
the connected neighborhoods $U\supset C$ narrow down.

\begin{theorem}\label{theorem}
Suppose that $F$ is a smooth and connected complex surface containing
a curve $C\cong\PP^1$ with positive self-intersection. Then
there exists a connected neighborhood $U\supset C$
such that for any connected neighborhood $V\supset C$, $V\subset U$,
the embedding of fields $\M(U)\hookrightarrow \M(V)$ induced by the
inclusion $V\subset U$ is an isomorphism.

In other words, there exists a connected neighborhood $U\supset C$
such that, for any connected neighborhood $V\supset C$, $V\subset U$,
any meromorphic function on~$V$ can be extended to a meromorphic
function on~$U$.
\end{theorem}

\begin{Not}
If $F$ is a smooth and connected complex surface containing
a curve $C\cong\PP^1$ with positive self-intersection, then
  \begin{equation*}
\tau(F,C)=\max_{U\supset C}\trdeg_\C\M(U),    
  \end{equation*}
where maximum is taken over all connected neighborhoods $U\supset C$. 
\end{Not}
It is clear that $\tau(F,C)\le 2$ and that $\tau(U,C)\le \tau(V,C)$ if
$U\supset V\supset C$.

Of course, Theorem~\ref{theorem} is trivial if $\tau(F,C)=0$, for in
this case all the fields involved consist entirely of constants. The
overall idea of the proof of Theorem~\ref{theorem} is the same for both
nontrivial cases $\tau(F,C)=1$ and $\tau(F,C)=2$, but
technicalities differ significantly.

The paper is organised as follows. In Section~\ref{sec:lemma} we
prove, for further reference, 
a folklore lemma that will be used in the proof for both cases. In
Section~\ref{tau=1} we prove Theorem~\ref{theorem} in the
case~$\tau(F,C)=1$. In Section~\ref{sec-gen-mer} we recall some known
general properties of meromorphic mappings. In Section~\ref{sec:ag} we
prove an auxiliary result on birational isomorphisms of rational
surfaces; the content of this section belongs entirely to algebraic
geometry. Finally, in Section~\ref{tau=2} we treat the case
$\tau(F,C)=2$, thus completing the proof of Theorem~\ref{theorem}. 

\subsection{Notation and conventions}
All topological terms refer to the classical topology unless specified
otherwise.

If $F$ is a smooth and connected complex manifold, then $\M(F)$ denotes
the field of meromorphic functions on~$F$.

By a \emph{curve} on a complex manifold we will mean its closed
one-dimensional irreducible (and reduced) complex subspace.

If $C$ is a \emph{compact} curve on a complex manifold~$F$,
$U_1,U_2\supset C$ are its neighborhoods, and if $B_1\subset U_1$,
$B_2\subset U_2$ are two curves meeting~$C$, we will say that
\emph{$B_1$ and $B_2$ have the same germ along~$C$} if there exists a
neighborhood $V\supset C$, $V\subset U_1\cap U_2$, such that $B_1\cap
V=B_2\cap V$.

The $n$-dimensional complex projective space is denoted by~$\PP^n$; in
particular, $\PP^1$ is the Riemann sphere.

If $P\in\C[\lst Xn]$ is a polynomial in $n$ variables, then
$V(P)\subset\C^n$ is its zero locus.

In diagrams, solid arrows refer to holomorphic mappings
and dashed arrows refer to meromorphic mappings
(see Section~\ref{sec-gen-mer}). 

In arguments where the only complex spaces involved are complex
projectve varieties, the terms `meromorphic mapping' and `rational
mapping' will be used interchangeably, and the same applies to the
pair `meromorphic function\slash rational function'.

If a field~$K$ contains~\C as a subfield, then by generators of~$K$ we
will always mean generators over~\C.

\section{A simple lemma}\label{sec:lemma}

\begin{lemma}\label{zeroes/cap_C}
Suppose that $F$ is a smooth and connected analytic surface containing
a curve $C\subset F$, $C\cong\PP^1$, with positive
self-intersection. If $f$ is a non-constant meromorphic function
on~$F$, then its divisor of zeroes \textup(as well as its divisor of
poles\textup) has non-empty intersection with~$C$.  
\end{lemma}

\begin{proof}
Observe first that a holomorphic function on any connected
neighborhood of~$C$ must be constant. Indeed, if $U\supset C$ is such
a neighborhood and $f\colon U\to \C$ is holomorphic, then it suffices
to show that $f$ is constant on $V\subset U$, where $V$ is a good
neighborhood of~$C$ in the sense of \cite[Definition 4.1]{field}. The
good neighborhood~$V$ is swept by curves isomorphic to~$C$, so $f$ is
contant on each of these curves, and any two such curves have
non-empty intersetion, so all these constants are the same. 

Assume now that the divisor of zeroes of~$f$ does not
intersect~$C$; then there exists a connected neighborhood $W\supset
C$, $W\subset F$, such that $f$ has no zeroes on~$W$. Then the
function $1/f$ is holomorphic on $W$, hence it is constant
on~$W$, hence $1/f$ is constant on~$F$, hence $f$ is
constant on~$F$, a contradiction.

The assertion about poles of~$f$ is obtained by substituting
$1/f$ for $f$ in the argument above.
\end{proof}

\section{The $\tau(F,C)=1$ case}\label{tau=1}

In this section we account for the case~$\tau(F,C)=1$ in
Theorem~\ref{theorem}. To that end it suffices to proof the following.

\begin{proposition}\label{prop:trdeg1}
Suppose that $F$  is a smooth and connected complex surface containing
a curve $C\cong\PP^1$ with positive self-intersection and that, for
any connected neighborhood $U\supset C$, $\trdeg_\C\M(U)=1$. Then there
exists a connected neigborhood $U\supset C$ such that for any
open and connected~$V\supset C$, $V\subset U$, the inclusion
$V\hookrightarrow U$ induces an isomorphism of fields of meromorphic
functions. 
\end{proposition}

Observe that if the hypotheses of Proposition~\ref{prop:trdeg1} are
satisfied, then, according to~\cite{field}, for any open and connected
$U\supset C$ the field $\M(U)$ is isomorphic to the field~$\C(T)$ of
rational functions in one variable, so $\M(U)$ is generated by one
element.

Further on, we will say that a meromorphic function $f$ on a
surface~$F$ is constant on a curve~$C\subset F$ if either $f$ is
constant or there exists an~$\alpha\in\C$ such that the
function~$f-\alpha$ has a zero along~$C$.

\begin{proposition}\label{two-kinds}
Suppose that $F$ is a smooth and connected complex surface containing
a curve $C\cong\PP^1$ with positive self-intersection and that
$\trdeg_\C\M(F)=1$.  Then the following assertions are equivalent.

\begin{enumerate}
\item\label{b1} For any meromorphic function~$f$ that is a generator of
$\M(F)$, $f$ is 
either constant on~$C$ or has a pole along~$C$.

\item\label{b2}
There exists a generator $f$ of $\M(F)$ such that $f$
has a zero along~$C$. 

\item\label{b3}
There exists an $h\in \M(F)$  such that $h$ has a 
pole along~$C$.

\item\label{b4}
There exists a non-constant $h\in \M(F)$  such that $h$ is
constant on~$C$.
\end{enumerate}
\end{proposition}

\begin{proof}
The implications $\eqref{b1}\Rightarrow \eqref{b2}\Rightarrow
\eqref{b3}\Rightarrow \eqref{b4}$ are obvious.  To prove the
implication $\eqref{b4}\Rightarrow \eqref{b1}$, suppose that the
restriction to~$C$ of a meromorphic function~$h$ is identically equal
to~$\alpha$. Then $h-\alpha=R(f)$, where $f$ is a generator of
the field~$\M(F)$ and~$R$ is a non-constant rational function. Thus,
  \begin{equation}\label{1}
h-\alpha=\frac{(f-a_1)\cdot\dots\cdot(f-a_m)}{Q(f)},
  \end{equation}
where \lst am are complex numbers and $Q\in\C[X]$ is a
polynomial such that none of \lst am is its root. (The product in the
numerator in the right-hand side of~\eqref{1} may be empty, in which
case the numerator is meant to be identically~$1$.) Since the
left-hand side of~\eqref{1} is identically zero on~$C$, the
function~$f$ is identically equal to $a_j$ on~$C$, for some~$j$.

If the numerator in~\eqref{1} is identically~$1$, then, obviously,
$\alpha=0$ and $f$ has a pole along~$C$
\end{proof}

\begin{definition}\label{def:kinds}
If the equivalent assertions of
Proposition~\ref{two-kinds} are satisfied for a surface~$F$ and a
curve $C\cong\PP^1$ contained in~$F$, we will say that the
pair~$(F,C)$ is \emph{of the first kind}. Otherwise, we will say
that $(F,C)$ is \emph{of the second kind}. 
\end{definition}

\begin{proposition}
If, for any neighborhood $U\subset C$, the pair~$(U,C)$ is of the
second kind in the sense of 
Definition~\ref{def:kinds}, and if $\tau(F,C)=1$, then
Proposition~\ref{prop:trdeg1} holds for~$(F,C)$.
\end{proposition}

\begin{proof}
Let $K$ stand for the field of meromorphic functions on~$C\cong\PP^1$;
$K$ is isomorphic to the field of rational functions in one variable
over~\C. Since, by the hypothesis, for any neighborhood $U\supset C$, the pair $(U,C)$ is
of the second kind, for any meromorphic function~$f\in\M(U)$ the
restriction $f|_C$ is well defined; this restriction induces the
embedding $\M(U)\hookrightarrow K$. Thus, for any connected
neighborhoof $U\supset C$ one has
\begin{equation*}
\M(F)\subset \M(U)\subset K.  
\end{equation*}
Since both $K$ and $\M(F)$ are finitely generated extensions of~\C, of
transcendence degree~$1$, one has $[K:\M(F)]<\infty$, so there exists
a neighborhood $U\supset C$ for which $[\M(U):\M(F)]$ is maximal. It
is clear that, once $U\supset V\supset C$, where $V$ is open and
connected, one has $\M(V)=\M(U)$, so the neighborhood~$U$ is the
neighborhood the existence of which we are to prove.
\end{proof}

In view of what we have just proved, to finish the proof of
Proposition~\ref{prop:trdeg1} it remains to account for the case in
which there is a connected neighborhood $U\supset C$ such that the
pair~$(U,C)$ is of the first kind; this is what we will assume from
now on. It is clear from Proposition~\ref{two-kinds} that in this case
the pair $(V,C)$ is also of the first kind for any open and
connected~$V$ such that $U\supset V\supset C$.

\begin{proof}[End of the proof of Proposition~\ref{prop:trdeg1}]
Arguing by contradiction, suppose that there is an infinite
sequence of nested connected neighborhoods
\begin{equation*}
U_0\supset U_1\supset U_2\supset\dots \supset U_n\supset\dots\supset C
\end{equation*}
such that $\M(U_j)\subsetneqq \M(U_{j+1})$ for each~$j$ and all
the~$U_j$ are of the first kind.  For each~$j$, choose once and for
all a meromorphic function $f_j$ on $U_j$ such that $f_j$ has a zero
along~$C$ and is a generator of the field~$\M(U_j)$; this is possible
in view of Proposition~\ref{two-kinds} and our assumption.

Each
embedding of fields $\M(U_j)\hookrightarrow \M(U_{j+1})$ corresponds
to a holomorphic mapping of algebraic curves
$\pi_j\colon C_{j+1}\to C_j$, where $C_j\cong C_{j+1}\cong \PP^1$
and $\deg\pi_j=[\M(U_{j+1}):\M(U_j)]>1$. This mapping is defined as follows:
if
$f_j|_{U_{j+1}}=R(f_{j+1})$, where $R$ is a rational function in one
variable, then the mapping $\pi_j$ is defined by the formula $z\mapsto
R(z)$. (I would like to stress that
the curves~$C_j$ are \emph{not} embedded in~$F$.)

The choice of $f_j$ in each $\M(U_j)$ allows one to identify
each~$C_j$ with the Riemann sphere $\PP^1=\C\cup\{\infty\}$: each
point $x\in C_j$ is identified with $T_j(x)\in \PP^1$, where $T_j$ is
the generator of the field~$\C(T_j)\cong \M(C_j)$ that corresponds
to~$f_j\in\M(U_j)$.

For each $j\ge 0$, put $\deg\pi_j=\delta_j>1$ and
$\ph_j=\pi_0\circ\dots\circ\pi_j$.
Since the field~\C is uncountable, there exists a point $a\in C_0$,
$a\ne\infty$, such that, for any $j\geqslant 0$, any point in
$(\pi_0\circ\dots\circ\pi_j)^{-1}(a)\subset C_{j+1}$ is not a
ramification point of the mapping $\ph_j\colon C_{j+1}\to C_0$ and is
different from $\infty\in C_j$. 

Put $m_j=\delta_0\delta_1\dots \delta_j$. By our choice of the
point~$a$, one has $\card (\ph_j)^{-1}(a)=m_j$.  If
$(\ph_j)^{-1}(a)=\{\lst b{m_j}\}$, then the image of the element
$f_0-a$ in $\M(C_{j+1})=\C(T_{j+1})$ has simple zeroes at the points
\lst b{m_j} and only at these points.  Recalling that
$\M(C_i)\cong\M(U_i)$ for each~$i$ and the embeddings
$\M(C_i)\hookrightarrow \M(C_{i+1})$ are induced by the restriction of
meromorphic functions from $U_i$ to~$U_{i+1}$, we conclude that on
$U_{j+1}$ the identity
\begin{equation}\label{eq:fact}
f_0-a=
\frac{(f_{j+1}-b_1)\cdot\dots\cdot (f_{j+1}-b_{m_j})}{Q(f_{j+1})},  
\end{equation}
holds, where $Q\in\C[X]$ is a polynomial such that none of \lst b{m_j} is
its root. 

It follows from~\eqref{eq:fact} that for each natural $k\le m_j$, the
divisor of zeroes of the function $f_0|_{U_{i+1}}-a$ contains the
divisor of zeroes of the function~$f_{j+1}-b_k$. All these divisors
are different, and
Lemma~\ref{zeroes/cap_C}
implies that each of them  has non-empty
intersection with~$C$. 

Hence, the divisor of zeroes of the meromorphic function $f_0-a$
contains at least $m_j$ different components, and each of these
components has non-empty intersection with~$C$. Since $C$ is compact,
the number of these components must be finte, but
$m_j$ tends to infinity as $j$ tends to
infinity. This contradiction completes the proof.
\end{proof}

\section{Generalities on rational and meromorphic
  mappings}\label{sec-gen-mer}

This section does not contain new results. We just state some
well-known facts about the ramification of meromorphic mappings, in the
form that suits our purposes.

For the definition of meromorphic functions and meromorphic mappings
we refer the reader to~\cite[Section 3]{AndreottiStoll}. 

If $X$ is a smooth complex manifold and \lst[0]fn are meromorphic
functions on $X$, then, if not all these functions are identically
zero, the formula $x\mapsto (f_0(x):\dots:f_n(x))$ defines a
meromorphic mapping from $X$ to~$\PP^n$.

If $X$ is a complex manifold, $Y$ is a smooth projective variety, and
$\ph\colon X\rightarrow Y$ is a meromorphic mapping, then the
indeterminacy locus (in~\cite{AndreottiStoll}, indeterminacy locus is
called just `indeterminacy') of~$\ph$ is a complex subvariety~$I\subset X$ of
codimension at least~$2$; the meromorphic mapping~$\ph$ induces a genuine
holomorphic mapping from $X\setminus I$ to~$Y$.

If $X$ is a projective variety, then meromorphic functions on $X$ are
the same as rational functions on~$X$ in the sense of algebraic
geometry; if $X$ and $Y$ are non-singular projective varieties, then
meromorphic mappings from $X$ to~$Y$ are the same as rational mappings
from $X$ to $Y$.

Suppose that $X$ and~$Y$ are projective varieties and $\ph\colon
X\dasharrow Y$ is a rational, aka meromorphic, mapping from $X$ to
$Y$; the mapping $\ph$ is called \emph{dominant} if $\ph(X\setminus I)$, where
$I\subset X$ is the indeterminancy locus of~$\ph$,
is dense in~$Y$ in the classical topology (or, which is equivalent,
in the Zariski topology). There is a natural bijection between the set
of dominant rational mappings $\ph\colon X\dasharrow Y$ and the set of
embeddings of fields $\iota=\ph^*\colon \M(Y)\hookrightarrow \M(X)$ as
extensions of~\C.

If $\ph\colon X\dasharrow Y$ is a dominant meromorphic mapping of
projective varieties of equal dimension, with indeterminacy
locus~$I\subset X$, then its \emph{degree}~$\deg \ph$ is the
cardianlity of a general fiber of the mapping $\ph|_{X\setminus
  I}\colon X\setminus I\to Y$. Degree of $\ph$ equals the degree of
the field extension $\M(X)\supset\ph^*\M(Y)$.

From this point and to the end of this section, we will be speaking
solely about complex sufraces, even though much of what follows can be
restated for complex varieties of dimension greater than~$2$. 

If $X$ is a smooth and connected complex surface, $Y$ is a projective
variety, and $\ph\colon X\dasharrow Y$ is a meromorphic mapping, then
the indeterminacy locus of~$\ph$ is a discrete subset~$I\subset X$. If
$E\subset X$ is a curve (i.e., a one-dimensional irreducible complex
subspace), then by $\ph(E)$ we will mean $\ph(E\setminus I)$, where $I$
is the indeterminacy locus of~$\ph$.

The section of the line bundle $\omega_{X\setminus I}^{-1}\otimes
\ph^*\omega_Y$ which, in local coordinates, is the Jacobian determinant
of~$\ph$, will be denoted by~$\Jac(\ph)$. If $\Jac(\ph)$ is not
identically zero but has zeroes, then the closure of its zero locus is
an analytic subspace in~$X$ by virtue of the Remmert--Stein theorem
(see for example \cite[Section 3]{Bishop}) and all the irreducible
components of this closure are curves in $X$.

\begin{definition}\label{ram_index}
Suppose that $X$ is a smooth and connected complex surface, $Y$
is a projective surface, and $\ph\colon X\dasharrow Y$ is a
meromorphic mapping such that $\Jac(\ph)$ is not identically zero.

If $D\subset X$ is a curve such that $\ph(D)$ is not a point,
then the \emph{ramification
  index} of~$\ph$ along~$D$ is the number $1+\ord_D\Jac(\ph)$.

The ramification index of~$\ph$ along~$D$ will be denoted
by~$e(\ph,D)$. 
\end{definition}  

The following proposition is very well known in the context of
algebraic surfaces.

\begin{proposition}\label{mult.ram}
Suppose that $X$ is a smooth complex surface, $Y$ and $Z$ are
smooth projective surfaces, and $\ph\colon X\dasharrow Y$ and
$\psi\colon Y\dasharrow Z$ are meromorphic mappings such that neither
$\Jac(\ph)$ nor $\Jac(\psi)$ is identically zero.

If $D_1\subset X$ a curve such
that $\ph(D_1)$ is not a point, $\ph(D_1)\subset D_2$, where
$D_2\subset Y$ is an irreducible curve, and $\psi(D_2)\subset Z$ is
not a point either, then
\begin{equation*}
e(\psi\circ\ph,D_1)=e(\psi,D_2)\cdot e(\ph,D_1). 
\end{equation*}
\end{proposition}

\begin{proof}
We begin with a lemma the straightforward proof of which is left to
the reader.
\begin{lemma}\label{mult.ram.lemma}
Suppose that $X=\{(t_1,t_2)\in\C^2\colon |t_1|,|t_2|<\eps\}$ is a
bidisk; put $D=\{(0,t)\}\subset X$. If $\ph\colon X\to \C^2$is a
holomorphic mapping defined by the formula
\begin{equation*}
(z,w)\mapsto (t_1^k\ph_1(t_1,t_2),\ph_2(t_1,t_2)), 
\end{equation*}
where $k>0$ is an integer, the function $t\mapsto \ph_1(0,t)$ is not
identically zero, and the function $t\mapsto \ph_2(0,t)$ is not
constant, then $e(\ph,D)=k$ and $\Jac(\ph)=t_1^{k-1}F(t_1,t_2)$,
where $F(0,t)$ is not identically zero. \textup(Here, by $\Jac$ we
mean the ordinary Jacobian determinant.\textup)
\end{lemma}
Now put $e(\ph,D_1)=e_1$, $e(\ph,D_2)=e_2$.
Since, in local coordinates, $\Jac(\psi\circ\ph)=\Jac(\ph)\cdot
(\Jac(\psi)\circ\ph)$, we have, in view of the lemma, 
\begin{multline*}
e(\psi\circ\ph,D_1)=1+\ord_{D_1}\Jac(\ph)+\ord_{D_1}(\Jac(\psi)\circ\ph)\\
{}=1+e_1-1+e_1(e_2-1)=e_1e_2,  
\end{multline*}
as required.
\end{proof}

\begin{definition}\label{def:strong_ram}
Suppose that $X$ is a smooth complex surface, $Y$ is a
smooth projective surface, and $\ph\colon X\dasharrow Y$ is a
meromorphic mapping such that $\Jac(\ph)$ is not identically zero. We
will say that $\ph$ is \emph{strongly ramified} along an irreducible
curve~$D\subset X$ if $\ph(D)$ is not a
point and $e(\ph,D)>1$.   
\end{definition}

\begin{note}
According to standard definitions, the ramification locus of~$\ph$ is
(the closure of) the zero locus of $\Jac(\ph)$. The aim of
Definition~\ref{def:strong_ram} is to rule out the components of the
ramification locus that are mapped to points.
\end{note}

\section{An auxiliary result from algebraic geometry}\label{sec:ag}

The results of this section are valid, with their proofs, over any
algebraically closed field, of arbitrary characteristic.

First we recall definitions and establish some notation concrning
quadratic transformations. Folowing~\cite{Dolgachev}, we will say that
the \emph{standard quadratic transformation}
$\Tst\colon\PP^2\dasharrow\PP^2$ is the birational automorphism of
$\PP^2$ defined by the formula
\begin{equation*}
(x_0:x_1:x_2)\mapsto (x_1x_2:x_0x_2:x_0x_1).  
\end{equation*}
If $(a,b,c)$ is any triple of non-collinear points of $\PP^2$, then by
$T_{abc}$ we will denote the birational automorphism of $\PP^2$
defined by the formula 
\begin{equation}\label{eq:Tabc}
T_{abc}=A_{abc}^{-1}\circ\Tst\circ
A_{abc}, 
\end{equation}
where $A\colon\PP^2\to \PP^2$ is a linear automorphism of
$\PP^2$ mapping the points $a$, $b$, and~$c$ to $(1:0:0)$, $(0:1:0)$,
and~$(0:0:1)$ respectively.

In the paragraph above we said ``\emph{a} linear automorphism'' since
the authomorphism $A_{abc}$ is not unique. Thus, the notation
$T_{abc}$ does not refer to a sole birational authomorphism: any two
such birational authomorphisms differ by a linear authomorphism of
$\PP^2$ fixing the points~$a$, $b$, and~$c$. To avoid ambiguity, we
choose, once and for all, one authomorhism $A_{abc}$ for any
non-collinear triple~$(a,b,c)$ and say that \emph{the}
authomorphism~$T_{abc}$ is defined by the formula~\eqref{eq:Tabc} for
this particular~$A_{abc}$. For any choice of the linear
automorphism~$A_{abc}$ the authomorphism~$T_{abc}$ blows up the
points~$a$, $b$, and~$c$ and blows down the lines $\overline{ab}$,
$\overline{bc}$, and~$\overline{ca}$. See \cite[Chapter 7]{Dolgachev}
or \cite[Section 1.4]{Cal-Cil} for more details.

We begin with a simple lemma.

\begin{lemma}\label{quadratic0}
Suppose that $\ph\colon \PP^2\dasharrow S$  is a birational isomorphism
between $\PP^2$ and a smooth projective surface $S$ such that $\ph$
induces and isomorphism between Zariski open subsets $U\subset\PP^2$
and~$V\subset S$. 

Suppose that $a\in U$. Put $p=\ph(a)$, and let $\sigma\colon \bar
S\to S$ be the blow-up
of~$S$ at~$p$; let $E\subset \bar S$ be the exceptional curve. 

If the points $b,c\in\PP^2$ are such that $a\notin\overline{bc}$, then
the birational mapping $\psi=\sigma^{-1}\circ \ph\circ T_{abc}\colon
\PP^2\dasharrow \bar S$ has the following properties.
\begin{enumerate}
\item All the line~$\overline{bc}$ except for the points~$b$ and~$c$
  lies in the domain of definition of the rational mapping~$\psi$.
\item The rational mapping~$\psi$ induces an isomorphism between~the
  line~$\overline{bc}$ and the exceptional curve~$E\subset \bar S$.
\item\label{5.1.3}
The above isomorphism maps the point~$b$ \textup(resp.~$c$\textup) to
the point of~$E$ corresponding to the tangent at the point~$p$ to
$\ph(\overline{ac})$ \textup(resp.~$\ph(\overline{ab})$\textup).
\end{enumerate}
\end{lemma} 

\begin{note}
The swap of $b$ and~$c$ in assertion~\eqref{5.1.3} above is not a typo.
\end{note}

\begin{proof}
Let 
$\sigma_0\colon \overline {\PP^2}\to\PP^2$
be the blowup of~$\PP^2$ at the point $a$
and put $\bar\ph=\sigma^{-1}\circ\ph\circ\sigma_1$ (see the
diagram below).
\begin{equation*}
\xymatrix{
&{\overline{\PP^2}\ar[d]^{\sigma_0}}\ar@{-->}[r]_{\bar\ph}&{\bar
    S}\ar[d]^\sigma\\
{\PP^2}\ar@{-->}[r]_{T_{abc}}
   \ar@/^3pc/@{-->}[urr]^{\psi}&
{\PP^2}\ar@{-->}[r]_\ph&S
}  
\end{equation*}
It is clear that $\bar\ph$ maps a Zariski neighborhood of
$\sigma_0^{-1}(a)$ isomorphically to a Zariski neighborhood of~$E$; hence, it
suffices to check the lemma for the case $F=\PP^2$ and
$\ph=\mathrm{id}$; this amounts to a straightforward computation that
is left to the reader. 
\end{proof}

\begin{proposition}\label{quadratic1}
Let $\pi\colon X\to\PP^2$ be a birational morphism, where $X$ is a
smooth projective surface.

If $p\in\PP^2$ is a point such that $\dim\pi^{-1}(p)>0$, then there
exists a birational automorphism $\chi\colon \PP^2\dasharrow \PP^2$
such that for any irreducible curve $E\subset \pi^{-1}(p)$ there is an
irreducible curve $D\subset\PP^2$ for which
$(\pi^{-1}\circ\chi)(D)=E$.
\end{proposition}

\begin{proof}
It follows from the hypothesis that $X$ is obtained from $\PP^2$ by a
sequence of blowups of points and $\sigma$ is the composition of the
corresponding blowdowns, $\pi=\sigma_1\circ\ldots\circ \sigma_n$ as
in~\eqref{eq:sigmas},
\begin{equation}\label{eq:sigmas}
X=X_n\xrightarrow{\sigma_n} X_{n-1}\xrightarrow{\sigma_{n-1}}\dots
\xrightarrow{\sigma_2}X_1  \xrightarrow{\sigma_1}X_0=\PP^2.
\end{equation}
where each $\sigma_j$ is the blowup of the point~$p_{j-1}\in X_{j-1}$ (we
put~$p_0=p$); put $E_j=\sigma_j^{-1}(p_{j-1})\subset X_j$.

It suffices to prove the proposition for the case in which
$p\in\PP^2$ is the only point the fiber of~$\pi$ over which has
positive dimension, and this is what we will assume further on. 

For any $j\ge k>0$ we denote by $\tilde E_{k,j}\subset X_j$ the strict
transform of $E_k$ with respect to the birational morphism
$\sigma_{k+1}\circ\ldots\circ \sigma_j\colon X_j\to X_k$; if $k=j$, we
put $\tilde E_{j,j}=E_j$. Finally, for any $i$ put 
\begin{equation}\label{eq:S_j}
Z_i=(\sigma_{i+1}\circ\ldots\circ \sigma_n)\left(\tilde
E_{i,n}\cap\bigcup_{j\ne i}\tilde E_{j,n}\right)\subset E_i.
\end{equation}
Now we construct inductively non-collinear triples of
points~$(a_j,b_j,c_j)$ of $\PP^2$, $0\le j< n$ and birational
isomorphisms $\ph_j\colon \PP^2\dasharrow X_j$, $0\le j\le n$, 
having the
following properties:
\begin{enumerate}
\item the diagram
  \begin{equation*}
\xymatrix{
{X_n}\ar[r]^{\sigma_n}&{\dots}&{X_2}\ar[r]^{\sigma_2}&{X_1}\ar[r]^{\sigma_1}&{X_0}
                           \ar@{=}[r]&{\PP^2}\\
{\PP^2}\ar@{-->}[u]_{\ph_n}\ar@{-->}[r]^{T_{n-1}}&{\dots}&
   {\PP^2}\ar@{-->}[u]_{\ph_2}\ar@{-->}[r]^{T_1}&
   {\PP^2}\ar@{-->}[u]_{\ph_1}\ar@{-->}[r]^{T_0}
&{\PP^2,}\ar@{-->}[u]_{\ph_0}   
} 
  \end{equation*}
where $T_j=T_{a_jb_jc_j}$,
is commutative;
\item
for each~$j\ge k>0$, the birational isomorphism $\ph_j$ maps the line
$\overline{b_kc_k}\subset\PP^2$ isomorphically to $\tilde
E_{j,k}\subset X_j$; 
\item
if $h\colon\overline{b_kc_k}\to \tilde E_{j,k}$ is the holomorphic mapping
induced by the rational mapping
$\sigma_{k+1}\circ\ldots\circ\sigma_j\circ \ph_j\colon \PP^2\dasharrow
X_k$, then $h(\{b_k,c_k\}))\cap Z_k=\varnothing$.
\end{enumerate}
To that end, first we put $\ph_0=\mathrm{id}_{\PP^2}$ and
$a_0=p$. Choose the points $b_0,c_0\in\PP^2$ so that the points
on~$E_1$ correponding to the directions of lines $\overline{a_0b_0}$
an $\overline{a_0c_0}$ differ from all the points of the finite set
$Z_1\subset E_1$ defined in~\eqref{eq:S_j}.  Put $T_0=T_{a_0b_0c_0}$,
$\ph_1=\sigma_1^{-1}\circ\ph_0\circ T_0$; by virtue of
Lemma~\ref{quadratic0}, the birational isomorphism~$\ph_0$ has the
required properties.

\begin{figure}
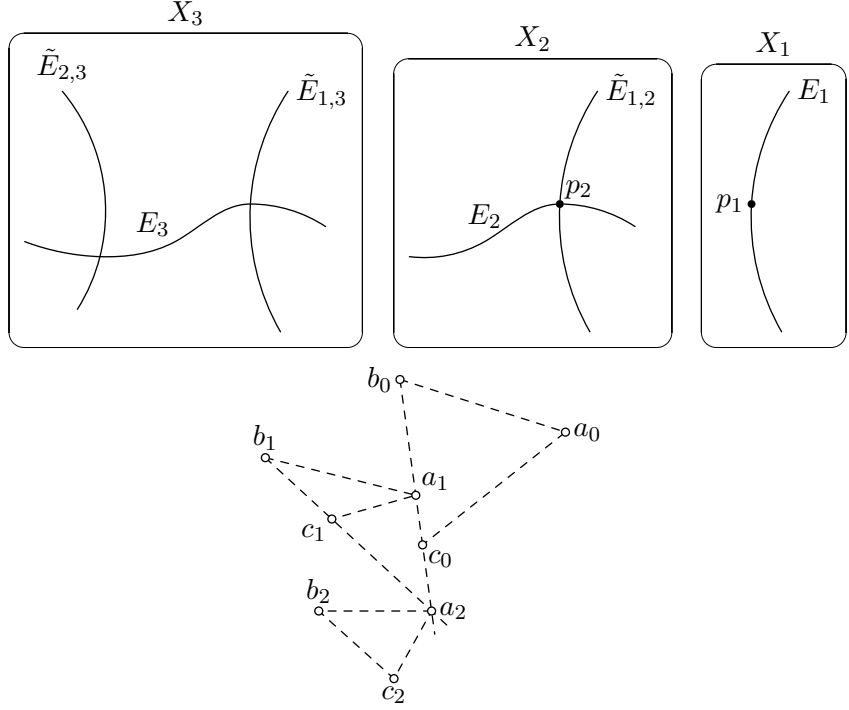

  \begin{tabular}{ccc}
\includegraphics{quad3.mps}&    
\includegraphics{quad2.mps}&    
\includegraphics{quad1.mps}\\[\smallskipamount]
\multicolumn{3}{c}{\includegraphics{quad4.mps}}    
  \end{tabular}
  \caption{A sequence of three blow-ups of $\PP^2$ and a sequence of
    quadratic transformations corresponding to them}\label{figure}
\end{figure}

If the points $a_k$, $b_k$, $c_k$, $0\le k<j$, and the birational
isomorphisms $\ph_k$, $0\le k\le j$, are already constructed, we
proceed as follows. Suppose that the point $p_j\in X_j$ that was blown
up to obtain $X_{j+1}$ lies on the curve $\tilde E_{k,j}\subset X_j$,
$k\le j$. By the induction hypothesis and Lemma~\ref{quadratic0},
$\ph_j(\overline{b_kc_k})=\tilde E_{j,k}$ and $\ph_j$ is a regular
isomorphism onto a Zariski neighborhood of~$p_j$. Now let~$a_j$ be the
point of $\overline{b_kc_k}$ that is mapped to~$p_j$ by~$\ph_j$; we
choose the points $b_j,c_j\in\PP^2$ so that the lines
$\overline{a_j,b_j}$ and $\overline{a_j,c_j}$ are different from the
lines $\overline{b_k,c_k}$ for~$k<j$ and the points on~$E_{j+1}$
correponding to the directions of lines $\overline{a_jb_j}$ an
$\overline{a_jc_j}$ differ from all the points of the finite set
$Z_{j+1}\subset E_{j+1}$.  If we put
$\ph_{j+1}=\sigma_{j+1}\circ\ph_j\circ T_j$, where
$T_j=T_{a_jb_jc_j}$, the induction step will be completed. This
process is illustrared in Fig.~\ref{figure}, for a certain sequence of
three blow-ups of the plane.

If we put $\chi=T_0\circ\ldots\circ T_{n-1}$, and if $E$ is a
component of $\sigma^{-1}(p)$, then~$E=\tilde E_{j,n}$ for some~$j$
and, for an appropriate line $L=\overline{a_kb_k}\subset \PP^2$,
$(\sigma^{-1}\circ\chi)(L)=\ph_n(L)=E$, as required.
\end{proof}

\section{The $\tau=2$ case}\label{tau=2}

We begin with an algebro-geometric lemma. It is valid over any
algebraically closed field of characterisitic zero, but we will not
pursue algebraic purity.

\begin{lemma}\label{discriminant}
Suppose that the polynomials $\lst[0]A{d-1}  \in\C[X,Y]$, where $d>1$,
are such that the 
polynomial 
\begin{equation}\label{P(X,Y,Z)}
P(X,Y,Z)=Z^d+A_{d-1}Z^{d-1}+\dots+A_0(X,Y) 
\end{equation}
is irreducible, and put $S=V(P)\subset\C^3$.

Then the following assertions hold.
\begin{enumerate}
\item\label{item6.1a}
The discriminant of the polynomial~$P$ with respect to~$Z$ is not
constant as a polynomial in~$X$ and~$Y$.

\item\label{item6.1b}
There exists a projective surface~$\hat S$ that is a desingularization
of the projective closure of~$S$ and such that if $p\colon
\hat S\dasharrow \PP^2$ is the rational mapping induced by $(X,Y,Z)\mapsto
(X,Y)$, then $p$ is strictly ramified along an irreducible curve
$\mathfrak D\subset \hat S$ 
for which $\overline{p(\mathfrak D)}\subset\PP^2$ is the projective
closure of a
component of the zero locus of the discrinminant of~$P$. 
\end{enumerate}
\end{lemma}

\begin{proof}
To prove assertion~\eqref{item6.1a}, assume, by way of contradiction, that the
discriminant is constant. Since $P$ is irreducible, this constant is
not zero, so the discriminant of~$P$ never vanishes. Hence the surface
$S$ is smooth and the projecton $p\colon S\to\C^2$
(``forgetting the $Z$ coordinate'') is a covering (with respect to the
classical topology); so, $S$ has $d>1$ connected components, which
is impossible since an irreducible algebraic variety over~\C must be connected in the classical topology (see~\cite[Chapter~VII, Section~2.2, Theorem~1]{Shafarevich}).

To prove~\eqref{item6.1b}, denote by $D\subset\C^2$ the zero locus of
the discriminant of~$P$. In view of what we have just proved, $D$ is a
plane algebraic curve.

Let $\nu\colon S_1\to S$ be the normalization of~$S$ and
put $\nu_1=p\circ\nu$, where $p$ is as above.
The
set~$\Sigma_1\subset S_1$ of singular points of~$S_1$ is finite. It is
clear that $\nu_1^{-1}(\C^2\setminus D)$ is smooth (in particular, $\Sigma=\nu_1(\Sigma_1)$ is contained in~$D$) and that the morphism
${\nu_1}|_{S_1\setminus\nu_1^{-1}(D)}
\colon S_1\setminus\nu_1^{-1}(D)\to\C^2\setminus D$
is unramified.

However, the morphism $\nu_1|_{S_1\setminus \nu_1^{-1}\Sigma}$ must be
ramified somewhere: if it were unramified, then, since
$\C^2\setminus\Sigma$ is simply connected, the irreducible variety
$S_1\setminus \nu_1^{-1}\Sigma$ would consist of $d$ conneted
components, which contradicts \cite[Chapter~VII, Section~2.2,
  Theorem~1]{Shafarevich} again. Thus, $\nu_1$ is ramified over a
curve that is mapped by $\nu_1$ to a component of the curve~$D$. To
complete the proof, it remains to embed $S_1$, as a Zariski open
subset, into a projective surface and take a desingularization~$\hat
S$ of the latter.
\end{proof}

From now on, $F$ will be a smooth and connected complex surface,
$F\supset C$, where $C\cong \PP^1$ and the self-intersection index
of~$C$ is positive, and $\trdeg_\C\M(F)=2$.  Since in this case
$\M(F)$ is, according to~\cite{field}, isomorphic to the field of rational functions~$\C(X,Y)$,
there exists a pair of generators $(f,g)$ of $\M(F)$; the meromorphic
functions $f$ and~$g$ are algebraically independent over~\C.

\begin{Not}\label{def:ph}
In the above setting, $\ph_{f,g}\colon F\dasharrow\PP^2$ is the
meromorphic mapping from $F$ to $\PP^2$ defined by $x\mapsto
(1:f(x):g(x))$.   
\end{Not}

Now let $U\supset V\supset C$ be connected neighborhoods
of~$C$ in~$F$. Since both $\M(U)$ and $\M(V)$ are isomorphic to the field of rational functions~$\C(X,Y)$, one may choose a pair of algebraically independent generators~$(f_0,g_0)$
for~$\M(U)$ and a pair of algebraically independent generators~$(f_1,g_1)$ for~$\M(V)$.

Let $\PP^2_0$ and $\PP^2_1$ be two copies of
$\PP^2$, and let $\pi\colon \PP^2_1\dasharrow
\PP^2_0$ be the rational (=meromorphic) mapping for which
diagram~\eqref{cd} is commutative
\begin{equation}\label{cd}
\xymatrix{
U\ar@{-->}[d]_{\ph_{f_0,g_0}} & V\ar@{_{(}->}[l]_i\ar@{-->}[d]^{\ph_{f_1,g_1}}\\
{\PP^2_0}& {\PP^2_1}\ar@{-->}[l]_\pi  
}
\end{equation}
(if  $\rho$ and~$\tau$ are rational functions in two variables such that
$f_0|_V=\rho(f_1,g_1)$, $g_0|_V=\tau(f_1,g_1)$, then $\pi$ maps a point with
homogeneous coordinates~$(1:x:y)$ to $(1:\rho(x,y):\tau(x:y))$).

\begin{proposition}\label{branching}
Suppose that, in the above setting, $\M(V)\supsetneqq\M(U)$. Then for
any pair $(f_0,g_0)$ of generators of~$\M(U)$ there exists a pair
$(f_1,g_1)$ of generators of~$\M(V)$, an irreducible curve
$\Delta\subset\PP^2_1$, and an irreducible curve \textup(one-dimensional
complex space\textup)~$B\subset V$ such that $\pi$ is strongly
ramified along~$\Delta$, $B\cap C\ne \varnothing$, $\ph_{f_1,g_1}(B)\subset
\Delta$, and $\ph_{f_0,g_0}|_V$ is strongly ramified along~$B$.

Besides, the ramification index $e(\ph_{f_0,g_0}|_V,B)$ is greater or
equal to $e(\pi,\Delta)$.
\end{proposition}

\begin{proof}
Put $d=[\M(V):\M(U)]>1$. 
The mapping $R\mapsto R(f_0,g_0)$, where $R$ is a rational function
in~$T_1$ and~$T_2$, defines an isomorphism $\C(T_1,T_2)\to\M(U)$, and we will
identify $\M(U)$ with $\C(T_1,T_2)$ via this isomorphism. One has
$\M(V)=\M(U)(h)$; multiplying~$h$, if necessary, by an appropriate
polynomial in~$f_0$ and~$g_0$, we may and will assume that $h$ satisfies
the equation
\begin{equation}\label{eqn-for-h:new}
h^d+A_{d-1}(f_0,g_0)h^{d-1}+\dots+A_0(f_0,g_0),  
\end{equation}
where \lst[0]A{d-1} are polynomials in two variables, with irreducible left-hand side.

Now we apply Lemma~\ref{discriminant} to the irreducible polynomial
\begin{equation}\label{eqn-for-Z}
P(X,Y,Z)=Z^d+A_{d-1}(X,Y)Z^{d-1}+\dots+A_0(X,Y);
\end{equation}
let $\hat S$, $p\colon \hat S\dasharrow
\PP^2_0$, and $\mathfrak D\subset \hat S$ be the objects the
existence of which is 
asserted by this lemma;  put $D_0=\overline{p(\mathfrak D)}\cap\C^2$, where $\C^2=\{(1:z:w)\}\subset\PP^2$.

Denote by $\psi\colon V\dasharrow \hat S$ the meromorphic mapping induced
by the meromorphic mapping $x\mapsto (1:f_0(x):g_0(x):h(x))$, from $V$
to~$\PP^3\supset\C^3$.  The formula $(1:z_1:z_2:z_3)\mapsto
(1:z_1:z_2)$ induces a dominant rational mapping $\hat S\dasharrow
\PP^2_0$, of degree~$d$. Let $\delta\in\C[X,Y]$ be an irreducible
equation of the curve~$D_0\subset\C^2$; put
$\mu=\delta(f_0,g_0)\in\M(U)$. Since $\delta$ is a non-constant
polynomial and the meromorphic functions $f_0$ and~$g_0$ are
algebraically independent over~\C,
the meromorphic function~$\mu$ is not
constant;
by Lemma~\ref{zeroes/cap_C}, there exists a
component $B\subset V$ of its divisor of zeroes such that $B\cap C\ne
\varnothing$; one has $\psi(B)\subset\Delta$.

By our construction, $\M(\hat S)\cong\M(V)$ is isomorphic to the field of
rational functions over~\C in two variables, so the surface~$\hat S$ is
rational. Thus, there exists a smooth projective surface $\hat{\hat S}$
and birational morphisms $\sigma\colon \hat{\hat S}\to \hat S$, $\sigma'\colon
\hat{\hat S}\to \PP^2_1$, where $\PP^2_1$ is another copy of~$\PP^2$ (see
diagram~\eqref{diagram2}).
\begin{equation}\label{diagram2}
\xymatrix{
U\ar@{-->}[d]_{\ph_{f_0,g_0}} & V\ar@{_{(}->}[l]_i\ar@{-->}[d]_\psi
   \ar@/^/@{-->}[drrr]^{\ph=\ph_{f_1,g_1}}\\
{\PP^2_0}& {\hat S}\ar@{-->}[l]_{\pi'}&{\hat{\hat S}}\ar[l]_\sigma\ar[r]^{\sigma'}&
{\PP^2_1}&{\PP^2_1}\ar@{-->}[l]_\chi\ar@/^2pc/@{-->}[llll]^{\pi}
}  
\end{equation}

Let $\Delta_1\subset \hat{\hat S}$ be the strict transform of
$\Delta\subset \hat S$
with respect to~$\sigma$. I claim that there exists a birational
isomorphism $\chi\colon \PP^2_1\dasharrow \PP^2_1$ such that
$(\chi^{-1}\circ \sigma')(\Delta_1)$ is a curve. Indeed, if $\Delta_1$
does not lie in a fiber of~$\sigma'$, one may just
put~$\chi=\mathrm{id}$, and if $\Delta_1$ lies in a fiber
of~$\sigma'$, it follows from Proposition~\ref{quadratic1} in which
one puts $X=\hat{\hat S}$ and $\pi=\sigma'$.

Now put $\ph=\chi^{-1}\circ\sigma'\circ \sigma^{-1}\circ\psi\colon
V\dasharrow \PP^2_1$. 
Since $\psi$, by construction, induces an isomorphism between $\M(V)$
and $\M(F)$, and since the arrows~$\sigma$, $\sigma'$, and~$\chi$
in~\eqref{diagram2} are birational, the meromorphic mapping~$\ph$
induces an isomorphism between~$\M(V)$ and $\M(\PP^2_1)$. If
$(w_0:w_1:w_2)$ are homogeneous coordinates on~$\PP^2_1$, then,
putting $f_1=(w_1/w_0)\circ\ph$, $g_1=(w_2/w_0)\circ\ph$, one sees
that $(f_1,g_1)$ is a pair of generators of $\M(V)$ and that
$\ph=\ph_{f_1,g_1}$. Removing the unnecessary objects and morphisms
from diagram~\eqref{diagram2}, we arrive at diagram~\eqref{cd}. 

Now let us put $D=(\chi^{-1}\circ\sigma')(\Delta_1)\subset
\PP^2_1$. By our construction, $\pi$ is strongly ramified along~$D$
and $\ph_{f_1,g_1}(B)\subset D$; hence, $\pi'\circ\psi$ is
strongly ramified along~$B$. On the other hand,
$\pi'\circ\psi=\ph_{f_0,g_0}\circ i$ and the embedding $i\colon
V\hookrightarrow U$ is unramified. Hence, $\ph_{f_0,g_0}|_V$ must be
ramified along~$B$.

Finally, it follows from the commutative diagram~\eqref{cd} and
Proposition~\ref{mult.ram} that
\begin{equation*}
e(\ph_{f_0,g_0}|_V,B)=e(\pi,D)\cdot
e(\ph_{f_1,g_1},B) \ge e(\pi,D),  
\end{equation*}
which proves the second assertion.
\end{proof}

Now we can account for the case $\tau(F,C)=2$ in
Theorem~\ref{theorem}. 

\begin{proposition}
Suppose that $F$ is a smooth connected projective surface such that
$\trdeg_\C\M(F)=2$. Then there exists a connected neighborhood $U\supset C$,
$U\subset F$ such that for any open and connected $V\subset U$,
$V\supset C$ the natural embedding $\M(U)\hookrightarrow \M(V)$ is
identical.
\end{proposition}

\begin{proof}
Arguing by contradiction, suppose that there exists an infinite
sequence of nested connected neighborhoods
\begin{equation*}
U_0\supset U_1\supset U_2\supset\dots \supset U_n\supset\dots\supset C
\end{equation*}
such that, for each~$j$, $\trdeg_\C(U_j)=2$ and $\M(U_j)\subsetneqq
\M(U_{j+1})$. For each $j\ge 0$, put $[\M(U_{j+1}):\M(U_j)]=d_j>1$.

According to~\cite{field}, each $\M(U_j)$ is isomorphic to the field
of rational functions over~\C in two variables. Using
Proposition~\ref{branching}, choose inductively, for any~$j\ge0$, a
pair of generators~$(f_j,g_j)$ for~$\M(U_j)$ such that the following
holds:
\begin{enumerate}
\item\label{a1} diagram~\eqref{large-cd}, in which
  $\ph_j=\ph_{f_j,g_j}$ and the rational mappings $\pi_j\colon
  \PP^2_{j+1}\dasharrow \PP^2_j$ come form the expression of
  $f_j|_{U_{j+1}}$ and $g_j|_{U_{j+1}}$ as rational functions of
  $f_{j+1}$ and $g_{j+1}$, is commutative;

\begin{equation}\label{large-cd}
  \xymatrix{
    U_0\ar@{-->}[d]^{\ph_0}&U_1\ar@{-->}[d]^{\ph_1}\ar@{_{(}->}[l]
    &U_2\ar@{-->}[d]^{\ph_2}\ar@{_{(}->}[l]&{\dots}\ar@{_{(}->}[l]
    &U_n\ar@{-->}[d]^{\ph_n}\ar@{_{(}->}[l]&{\dots}\ar@{_{(}->}[l]\\
  {\PP^2_0}&{\PP^2_1}\ar@{-->}[l]^{\pi_0}&{\PP^2_2}\ar@{-->}[l]^{\pi_1}
  &{\dots}\ar@{-->}[l]^{\pi_2}&{\PP^2_n}\ar@{-->}[l]^{\pi_{n-1}}
  &{\dots}\ar@{-->}[l]^{\pi_n}
  }
\end{equation}

\item\label{a2}
each $\pi_j$ is strongly ramified along an irreducible curve
$D_{j+1}\subset\PP^2_{j+1}$; 
\item\label{a3}
for each $j> 0$ there is an irreducible
curve $B_j\subset U_j$ such that $B_j\cap C\ne \varnothing$,
$\ph_j(B_j)$ is not a point, $\ph_j(B_j)\subset D_j$, and
$e(\ph_{j-1}|_{U_j},B_j)\ge e(\pi_{j-1},D_j)$. 
\end{enumerate}

For each natural $j>0$, put
\begin{equation*}
\pi_{0j}=\pi_0\circ \pi_{1}\circ\dots
\circ\pi_{j-1}\colon\PP^2_j\dasharrow\PP^2_0. 
\end{equation*}
It follows from assertion~\eqref{a3} above and
Proposition~\ref{mult.ram} that, for each $n>0$, $\ph_0|U_n$ is
strongly ramified along~$B_n$ and $e(\ph_0|U_n,B_n)\ge
e(\pi_{0n},D_n)$.

Since the curve $C$ is compact, the meromorphic mapping~$\ph_0$ cannot be
ramified along infinitely many different curves (one-dimensioanl
analytic subspaces) meeting~$C$. Hence, among the curves $B_n\subset
U_n$, $n>1$, infinitely many have the same germ along~$C$.  Removing
from the sequence $\{U_n\}$ the neighborhoods for which the germ
of~$B_n$ is different and renumbering the remaining neigborhoods
consecutively, we arrive at the diagram of the
form~\eqref{large-cd} for which assertions~\eqref{a1} and~\eqref{a2}
above hold, as well as the following modification of
assertion~\eqref{a3}:
\begin{itemize}
\item [($3'$)]
there is an irreducible
curve $B\subset U_0$ such that $B\cap C\ne \varnothing$ and, for any $j>0$,
$\ph_j(B\cap U_j)$ is not a point, $\ph_j(B\cap U_j)\subset D_j$, and
$e(\ph_{j-1}|_{U_j},B\cap U_j)\ge e(\pi_{j-1},D_j)$.
\end{itemize}

Since now, for any $j>0$, $\ph_j(B\cap U_j)\subset D_j$ and
$\ph_j(B\cap U_j)$ is not a point, we conclude that all the curves
$D_j$ are the same; denote this irreducible plane curve
by~$D$. Putting $e(\pi_j,D)=e_j>1$, one has, from 
diagram~\eqref{large-cd} and Proposition~\ref{mult.ram}, 
\begin{equation}\label{rhs->/infty}
e(\ph_0|_{U_n},B\cap U_n)\geqslant e_1\cdot\ldots\cdot e_n\quad\text{for any
  $n>0$.}   
\end{equation}
The right-hand side of \eqref{rhs->/infty} tends to infinity as
$n\to\infty$; since the ramification index of $\ph_0$ at a
given irreducible curve cannot be infinite, we arrived at a
contradiction.

This completes the proof of the proposition and of Theorem~\ref{theorem}.
\end{proof}

\bibliographystyle{amsplain}
\bibliography{stab}

\end{document}